\documentclass[a4paper,12pt]{amsart}
\usepackage[margin=3truecm,top=4truecm]{geometry}

\usepackage{amsthm,amsmath,amssymb,comment,setspace,mleftright}
\usepackage{graphicx,color}
\usepackage[dvips,colorlinks=true]{hyperref}

\allowdisplaybreaks 

\usepackage{lineno}
\newcommand*\patchAmsMathEnvironmentForLineno[1]{%
  \expandafter\let\csname old#1\expandafter\endcsname\csname #1\endcsname
  \expandafter\let\csname oldend#1\expandafter\endcsname\csname end#1\endcsname
  \renewenvironment{#1}%
     {\linenomath\csname old#1\endcsname}%
     {\csname oldend#1\endcsname\endlinenomath}}%
\newcommand*\patchBothAmsMathEnvironmentsForLineno[1]{%
  \patchAmsMathEnvironmentForLineno{#1}%
  \patchAmsMathEnvironmentForLineno{#1*}}%
\AtBeginDocument{%
\patchBothAmsMathEnvironmentsForLineno{equation}%
\patchBothAmsMathEnvironmentsForLineno{align}%
\patchBothAmsMathEnvironmentsForLineno{flalign}%
\patchBothAmsMathEnvironmentsForLineno{alignat}%
\patchBothAmsMathEnvironmentsForLineno{gather}%
\patchBothAmsMathEnvironmentsForLineno{multline}%
}

\makeatletter

\@addtoreset{equation}{section}
\makeatother

\newcounter{anc}

\newcommand{\C}{\refstepcounter{anc}{C_{\theanc}}}
\newcommand{\Cl}[1]{{\C\label{#1}}}
\newcommand{\Cr}[1]{{C_{\tiny \ref{#1}}}}

\theoremstyle{plain} 
\newtheorem{thm}{Theorem}[section]
\newtheorem{prop}[thm]{Proposition}
\newtheorem{cor}[thm]{Corollary}
\newtheorem{lem}[thm]{Lemma}

\theoremstyle{definition}

\theoremstyle{remark}

\newcommand{\N}{\mathbb{N}}
\newcommand{\R}{\mathbb{R}}
\newcommand{\Z}{\mathbb{Z}}
\newcommand{\Q}{\mathbb{Q}}

\renewcommand{\P}{\mathbb{P}}
\newcommand{\E}{\mathbb{E}}

\newcommand{\1}[1]{\mathbf{1}_{#1}}

\renewcommand{\hat}{\widehat}
\renewcommand{\bar}{\overline}

\newcommand{\hyphen}{\textrm{-}}
\newcommand{\as}{\textrm{a.s.}}

\newcommand{\Id}{\mathrm{Id}}

\newcounter{at}
\newcommand{\A}{\refstepcounter{at}{\alpha_{\theat}}}
\newcommand{\Al}[1]{{\A\label{#1}}}
\newcommand{\Ar}[1]{{\alpha_{\tiny \ref{#1}}}}

\begin{document}
\title[The first passage time in the frog model]
	{Deviation bounds for the first passage time\\ in the frog model}
\author[N.~KUBOTA]{Naoki KUBOTA}
\address[N. Kubota]
	{College of Science and Technology, Nihon University, Chiba 274-8501, Japan.}
\email{kubota.naoki08@nihon-u.ac.jp}
\thanks{The author was supported by JSPS Grant-in-Aid for Young Scientists (B) 16K17620.}

\begin{abstract}
We consider the so-called frog model with random initial configurations.
The dynamics of this model is described as follows: Some particles are randomly assigned on any site of the multidimensional cubic lattice.
Initially, only particles at the origin are active and these independently perform simple random walks.
The other particles are sleeping and do not move at first.
When sleeping particles are hit by an active particle,
they become active and start moving in a similar fashion.
The aim of this paper is to derive large deviation and concentration bounds
for the first passage time at which an active particle reaches a target site.
\end{abstract}

\maketitle

\section{Introduction}

\subsection{The model}
For $d \geq 2$, we write $\Z^d$ for the $d$-dimensional cubic lattice.
Let $\omega=(\omega(x))_{x \in \Z^d}$ be independent random variables with a common law on $\N_0:=\N \cup \{0\}$,
not concentrated in zero.
Furthermore, independently of $\omega$, let $(S_k(x,\ell))_{k=0}^\infty$, $x \in \Z^d$, $\ell \in \N$,
be independent simple random walks on $\Z^d$ with $S_0(x,\ell)=x$.
We now introduce the \emph{first passage time} $T(x,y)$ from $x$ to $y$ as follows:
\begin{align*}
	T(x,y):=\inf\Biggl\{ \sum_{i=0}^{m-1}\tau(x_i,x_{i+1}):
	\begin{minipage}{4truecm}
		$m \geq 1$,\\
		$x=x_0,x_1,\dots,x_n=y$
	\end{minipage}
	\Biggr\},
\end{align*}
where
\begin{align*}
	\tau(x_i,x_{i+1}):=\inf\{ k \geq 0:S_k(x_i,\ell)=x_{i+1} \text{ for some } 1 \leq \ell \leq \omega(x_i) \}
\end{align*}
with the convention that $\tau(x_i,x_{i+1}):=\infty$ if $\omega(x_i)=0$.
The fundamental object of study is the first passage time $T(0,x)$ conditioned on the event $\{ \omega(0) \geq 1 \}$.
Its intuitive meaning is as follows:
We now regard simple random walks as ``frogs'' and $\omega$ stands for an initial configuration of frogs,
i.e., $\omega(y)$ frogs sit on each site $y$ (there is no frog at $y$ if $\omega(y)=0$).
Suppose that the origin $0$ is occupied by at least one frog.
They are active and independently perform simple random walks,
but the other frogs are sleeping and do not move at first.
When sleeping frogs are attacked by an active one, they become active and start doing independent simple random walks.
Then, $T(0,x)$ describes the first passage time at which an active frog reaches a site $x$.

It is straightforward to check that the first passage time has subadditivity:
\begin{align*}
	T(x,z) \leq T(x,y)+T(y,z),\qquad x,y,z \in \Z^d.
\end{align*}
In addition, Alves et al.~\cite[Lemmata~2.2 and 2.3]{AlvMacPopRav01} proved that
there exist constants $0<\Cl{Alv1},\Cl{Alv2}<\infty$ and $0<\Al{Alv3}<1$ such that
for all $x \in \Z^d$ and $t \geq \|x\|_1^4$,
\begin{align}\label{eq:Alv_tail}
	P(T(0,x) \geq t|\omega(0) \geq 1) \leq \Cr{Alv1} e^{-\Cr{Alv2}t^\Ar{Alv3}},
\end{align}
where $\|\cdot\|_1$ is the $\ell^1$-norm on $\R^d$.
This guarantees the integrability of $T(0,x)$.
As we are now working with the i.i.d.~set-up, an application of the subadditive ergodic theorem
enables us to show the following proposition (see \cite[Theorem~1.1 and Steps~1--6 in Section~2]{AlvMacPopRav01} for more details).

\begin{prop}\label{prop:Alv_shape}
There exists a norm $\mu(\cdot)$ (which is called the time constant) on $\R^d$ such that
almost surely on the event $\{ \omega(0) \geq 1 \}$,
\begin{align*}
	\lim_{\|x\|_1 \to \infty} \frac{T(0,x)-\mu(x)}{\|x\|_1}=0.
\end{align*}
Furthermore, $\mu(\cdot)$ is invariant under permutations of the coordinates and under reflections in the coordinate hyperplanes,
and satisfies
\begin{align}\label{eq:tc_bound}
	\|x\|_1 \leq \mu(x) \leq \mu(\xi_1)\|x\|_1,\qquad x \in \R^d,
\end{align}
where $\xi_1$ is the first coordinate vector of $\R^d$.
\end{prop}

\subsection{Main results}
Our main results are the following upper large deviation bounds for the first passage time.
Throughout this paper, we write $\P:=P(\cdot|\omega(0) \geq 1)$ to shorten notation.

\begin{thm}\label{thm:right}
There exists a constant $0<\Al{right1}<1$ such that for all $\epsilon>0$,
\begin{align*}
	\limsup_{\|x\|_1 \to \infty} \frac{1}{\|x\|_1^\Ar{right1}}
	\log \P(T(0,x) \geq (1+\epsilon)\mu(x))<0.
\end{align*}
\end{thm}

\begin{thm}\label{thm:left}
If $E[\omega(0)]<\infty$, then there exists a constant $0<\Al{left1}<1$ such that for all $\epsilon>0$,
\begin{align*}
	\limsup_{\|x\|_1 \to \infty} \frac{1}{\|x\|_1^\Ar{left1}}
	\log \P(T(0,x) \leq (1-\epsilon)\mu(x))<0.
\end{align*}
\end{thm}

Our key tool to prove the above theorems is the modified first passage time defined as follows.
Denote by $\mathcal{I}$ the random set of all sites of $\Z^d$ which frogs initially occupy, i.e.,
\begin{align*}
	\mathcal{I}:=\{ x \in \Z^d:\omega(x) \geq 1 \}.
\end{align*}
For any $x \in \Z^d$, let $x^*$ be the closest point to $x$ in $\mathcal{I}$ for the $\ell^1$-norm,
with a deterministic rule to break ties.
Then, the modified first passage time $T^*(x,y)$ is given by
\begin{align}\label{eq:mft}
	T^*(x,y):=T(x^*,y^*).
\end{align}
By definition, the subadditivity is inherited from the original first passage time:
\begin{align*}
	T^*(x,z) \leq T^*(x,y)+T^*(y,z),\qquad x,y,z \in \Z^d.
\end{align*}
A particular difference between $T(x,y)$ and $T^*(x,y)$ is that $T(x,y)$ is inevitably equal to infinity if $\omega(x)=0$,
but $T^*(x,y)$ can avoid that situation.
Moreover, we can derive the following concentration inequality for $T^*(0,x)$.

\begin{thm}\label{thm:tscon}
Assume $E[\omega(0)]<\infty$.
For all $\gamma>0$, there exist constants $0<\Cl{scon1},\Cl{scon2},\Cl{scon3}<\infty$ and $0<\Al{scon4}<1$ such that
for all $\Cr{scon1}(1+\log\|x\|_1)^{1/\Ar{scon4}} \leq t \leq \gamma \sqrt{\|x\|_1}$,
\begin{align*}
	P\Bigl( |T^*(0,x)-E[T^*(0,x)]| \geq t\sqrt{\|x\|_1} \Bigr)
	\leq \Cr{scon2}e^{-\Cr{scon3}t^\Ar{scon4}}.
\end{align*}
\end{thm}

Theorem~\ref{thm:tscon} is not only of independent interest in view of the investigation of the modified first passage time,
but also plays a key role to obtain Theorem~\ref{thm:left} as mentioned in Subsection~\ref{subsec:org} below.

We finally discuss, briefly, lower large deviation bounds for the first passage time.
Let us first observe the deviation to the right of $T(0,x)$ from $\mu(x)$.
Consider the event $A$ that $\omega(\xi_1)=0$ and $S_\cdot(0,\ell)$, $1 \leq \ell \leq \omega(0)$, stay
inside the set $\{ 0,\xi_1 \}$ until time $\lceil (1+\epsilon)\mu(\xi_1) \rceil \|x\|_1$.
By \eqref{eq:tc_bound}, we have $T(0,x) \geq (1+\epsilon)\mu(x)$ on the event $A \cap \{ 0 \in \mathcal{I} \}$.
It follows that
\begin{align*}
	\P(A)
	&\geq P(\xi_1 \not\in \mathcal{I}) P(0 \in \mathcal{I})^{-1}
		\sum_{L=1}^\infty P(\omega(0)=L) (2d)^{-L\lceil (1+\epsilon)\mu(\xi_1)\rceil \|x\|_1}\\
	&= P(0 \not\in \mathcal{I})
		\E \Bigl[ (2d)^{-\omega(0)\lceil (1+\epsilon)\mu(\xi_1)\rceil \|x\|_1} \Bigr].
\end{align*}
Jensen's inequality proves
\begin{align*}
	\liminf_{\|x\|_1 \to \infty} \frac{1}{\|x\|_1}\log \P(T(0,x) \geq (1+\epsilon)\mu(x))
	\geq -\E[\omega(0)] \lceil (1+\epsilon)\mu(\xi_1)\rceil \log(2d).
\end{align*}
If $E[\omega(0)]=\infty$, then this lower bound has no meaning,
otherwise this suggests that the optimal speed of the right tail large deviation is between $\|x\|_1^\Ar{right1}$ and $\|x\|_1$.

We next treat the deviation to the left  of $T(0,x)$ from $\mu(x)$.
In the case where $\inf_{\|y\|_1=1}\mu(y)<(1-\epsilon)^{-1}$, there is a direction $y \in \Q^d$
such that $\|y\|_1=1$ and $\mu(y)<(1-\epsilon)^{-1}$.
One has $T(0,Ny) \geq N$ for all $N \in \N$ with $Ny \in \Z^d$, and hence
\begin{align*}
	\P(T(0,Ny) \leq (1-\epsilon)\mu(Ny))=0.
\end{align*}
In particular,
\begin{align*}
	\liminf_{\|x\|_1 \to \infty} \frac{1}{\|x\|_1} \log \P(T(0,x) \leq (1-\epsilon)\mu(x))=-\infty.
\end{align*}
On the other hand, in the case where $\inf_{\|y\|_1=1}\mu(y) \geq (1-\epsilon)^{-1}$,
we have $(1-\epsilon)\mu(x) \geq \|x\|_1$ for all $x \in \Z^d$.
Fix a self-avoiding nearest-neighbor path $(0=v_0,v_1,\dots,v_n=x)$ with minimal length $n=\|x\|_1$
and let $A'$ be the event that $S_k(0,1)=v_k$ for all $0 \leq k \leq n$.
It holds that $\P(A')=(2d)^{-\|x\|_1}$ and $T(0,x) \leq (1-\epsilon)\mu(x)$ on the event $A' \cap \{ 0 \in \mathcal{I} \}$.
Hence,
\begin{align*}
	\liminf_{\|x\|_1 \to \infty} \frac{1}{\|x\|_1}\log \P(T(0,x) \leq (1-\epsilon)\mu(x)) \geq -\log(2d),
\end{align*}
which tells us that in this case, the optimal speed of the left tail large deviation is between $\|x\|_1^\Ar{left1}$ and $\|x\|_1$.

Optimizing the speeds for the above large deviations may be difficult in general.
The first passage time depends on the propagation of active frogs, and the following consideration suggests that
this propagation is strongly related to the dimension $d$ and the law of the initial configuration $\omega$.
The range of the simple random walk grows sublinearly in $d=2$ but linearly in $d \geq 3$
over time (see \cite[pages~333, 338]{Hug95_book}).
This means that sleeping frogs are likely to awaken in $d \geq 3$ as compared with the situation in $d=2$.
Apart from the dimension $d$, the outbreak of active frogs may occur if each site of $\Z^d$ has plenty of frogs with high probability.
However, for example, that situation is unusual in the case where the initial configuration of frogs
obeys a Bernoulli distribution with very small parameters.
In any case, we do not have enough information to determine the optimal speeds for the right and left tail large deviations,
and would like to address these problems in future research.

\subsection{Earlier literature}
The frog model was originally introduced by Ravishankar, and its idea comes from the following information spreading.
Consider that every active frog has some information.
When it hits sleeping frogs, the information is shared between them.
Active frogs move freely and play a role in spreading the information.

The first published result on the frog model is due to Telcs--Wormald~\cite[Section~2.4]{TelWor99}
(In their paper, the frog model was called the ``egg model'').
They treated the frog model on $\Z^d$ with one-frog-per-site initial configuration,
and proved that it is recurrent for all $d \geq 1$, i.e., almost surely, active frogs infinitely often visit the origin.
(Otherwise, we say that the frog model is transient.)
This result proposed an interesting relationship between the strength of transience for a single random walk
and the superior numbers of frogs.

To observe this more precisely, Popov~\cite{Pop01} considered the frog model with Bernoulli initial configurations
and exhibited phase transitions of its transience and recurrence.
After that, Alves et al.~coped with that kind of problem for the frog model with random initial configuration and random lifetime,
see \cite{AlvMacPop02b,Pop03} for more details.
In particular, \cite{Pop03} is a nice survey on the frog model and presents several open problems.
It has also been a great help to recent progress on recurrence and transience for the frog model.
We refer the reader to \cite{DobPfe14,GanSch09,GhoNorRoi17,HofWei16,KosZer17} for the frog model on lattices,
\cite{DobGanHofPopWei17_arXiv,DobPfe14,GanSch09} for the frog model with drift on lattices,
and \cite{HofJohJun16,HofJohJun17_arXiv,HofJohJun17} for the frog model on trees.

On the other hand, there are few results for the first passage time and the time constant of the frog model
except for \cite{AlvMacPop02,AlvMacPopRav01,JohJun16_arXiv}.
(Recently, the first passage time is also studied in a Euclidean setting, see \cite{BecDinDurHuoJun17_arXiv}.)
However, in view of information spreading, it is important to investigate these quantities more precisely,
and Theorems~\ref{thm:right}, \ref{thm:left} and \ref{thm:tscon} above present non-trivial deviation bounds for the first passage time.

\subsection{Organization of the paper}\label{subsec:org}
Let us now describe how the present article is organized.
In Section~\ref{sect:pre}, for convenience, we summarize some notation and results for supercritical site percolation on $\Z^d$
and provide an upper tail estimate for the first passage time (see Proposition~\ref{prop:APtype} below).
In particular, as a consequence of Proposition~\ref{prop:APtype}, we obtain that
with high probability, each frog realizing $T(0,x)$ must find the next one
within the $\ell^1$-ball of radius much smaller than $\|x\|_1$ (see Corollary~\ref{cor:jump} below).
The estimates stated in Section~\ref{sect:pre} play a key role to prove Theorems~\ref{thm:right}, \ref{thm:left} and \ref{thm:tscon}.

The goal of Section~\ref{sect:right} is to prove Theorem~\ref{thm:right}.
We basically follow the strategy taken in \cite[Subsection~3.3]{GarMar07}.
Note that Proposition~\ref{prop:Alv_shape} suggests that if $N$ is large enough, then for each site $y \in \Z^d$,
it happens with high probability that $T(Ny,N(y+\xi)) \approx N\mu(\xi_1)$ for all $\xi \in \Z^d$ with $\|\xi\|_1=1$.
(However, $T(Ny,N(y+\xi))=\infty$ holds if $\omega(Ny)=0$.
To avoid this, we need to use the modified first passage time given by \eqref{eq:mft}.)
Such a site $y$ is called ``good'', and good sites induce a finitely dependent site percolation on $\Z^d$
with parameter sufficiently close to one (see Lemma~\ref{lem:good} below).
For simplicity, suppose that $x=n\xi_1$ and an arbitrary integer $n$ is much larger than $N$.
Results in Subsection~\ref{subsect:perc} below guarantee that the failure probability of the following event decays exponentially in $n$:
There exist good sites $y_1,\dots,y_Q$ such that
\begin{itemize}
	\item $Q \approx n/N$ and $\|y_q-y_{q+1}\|_1=1$ for all $1 \leq q \leq Q-1$,
	\item $\|Ny_1\|_1$ and $\|n\xi_1-Ny_Q\|_1$ are much smaller than $n$.
\end{itemize}
On this event,
\begin{align*}
	T(0,n\xi_1)
	&\leq T(0,Ny_1)+\sum_{q=1}^{Q-1}T(Ny_q,Ny_{q+1})+T(Ny_Q,x)\\
	&\approx T(0,Ny_1)+\mu(n\xi_1)+T(Ny_Q,x).
\end{align*}
We use the upper tail estimate (stated in Proposition~\ref{prop:APtype}) to control the first and third terms of the most right side,
and get the desired bound in the case $x=n\xi_1$.
A few additional works are needed to carry out the above argument uniformly in any direction $x$.

In Section~\ref{sect:right}, we begin with the proof of Theorem~\ref{thm:left}.
The left large deviation bound has been studied for the first passage time in the first passage percolation
and the chemical distance in the Bernoulli percolation, see \cite{Ahl15}, \cite{GarMar10} and \cite{Kes86}.
These are similar quantities to the first passage time in the frog model,
but the approaches taken in \cite{Ahl15}, \cite{GarMar10} and \cite{Kes86} do not work well in our setting.
The main difficulty is here that the first passage time in the frog model is regarded as
a long-range version of the first passage percolation on $\Z^d$
and depends on both simple random walks and random initial configurations.
This difficulty disturbs the use of a renormalization procedure and a BK-like inequality, which are key tools in the aforementioned articles.
To overcome this problem, we use the concentration inequality for $T^*(0,x)$ as follows.
Divide $T(0,x)-\mu(x)$ into three terms:
\begin{align*}
	&T(0,x)-\mu(x)\\
	&= \{ T(0,x)-T^*(0,x)\}+\{ T^*(0,x)-E[T^*(0,x)] \}+\{ E[T^*(0,x)]-\mu(x) \}. 
\end{align*}
From Lemma~\ref{lem:modify} below, $E[T^*(0,x)] \geq \mu(x)$ holds and the third term is harmless for the left tail.
The second term can be controlled once we get the concentration inequality for $T^*(0,x)$, which is Theorem~\ref{thm:tscon}.
Hence, in the proof of Theorem~\ref{thm:left}, we try to compare $T(0,x)$ and $T^*(0,x)$ on the event $\{ \omega(0) \geq 1 \}$
by using Corollary~\ref{cor:jump}.

The remainder of Section~\ref{sect:right} will be devoted to the proof of Theorem~\ref{thm:tscon}.
The proof is based on Chebyshev's inequality and exponential versions of the Efron--Stein inequality.
This approach has already been taken by Garet--Marchand~\cite[Section~3]{GarMar10}
to derive the concentration inequality for the chemical distance in the Bernoulli percolation.
However, their model is a nearest-neighbor case.
We cannot directly apply their method to our model and modify it in the proof of Theorem~\ref{thm:tscon}.

We close this section with some general notation.
Write $\|\cdot\|_1$ and $\|\cdot\|_\infty$ for the $\ell^1$ and $\ell^\infty$-norms on $\R^d$.
Denote by $\{ \xi_1,\dots,\xi_d\}$ the canonical basis of $\R^d$ and let $\mathcal{E}^d:=\{ \xi \in \Z^d:\|\xi\|_1=1 \}$.
For $i \in \{1,\infty\}$, $x \in \R^d$ and $r>0$, $B_i(x,r)$ is the $\ell^i$-ball in $\R^d$ of center $x$ and radius $r$, i.e.,
\begin{align*}
	B_i(x,r):=\{ y \in \R^d:\|y-x\|_i \leq r \}.
\end{align*}
Throughout this paper, we use $c$, $c'$, $C$, $C'$, $C_i$ and $\alpha_i$, $i=1,2,\dots$, to denote constants with
$0<c,c',C,C',C_i<\infty$ and $0<\alpha_i<1$, respectively.

\section{Preliminaries}\label{sect:pre}

\subsection{Supercritical site percolation}\label{subsect:perc}
Let $X=(X_v)_{v \in \Z^d}$ be a family of random variables taking values in $\{ 0,1 \}$.
This induces the random set $\{v \in \Z^d:X_v=1 \}$.
The \emph{chemical distance} $d_X(v_1,v_2)$ for $X$ between $v_1$ and $v_2$ is defined by
\begin{align*}
	d_X(v_1,v_2):=\inf \biggl\{ \#\pi:
	\begin{minipage}{7.5truecm}
		$\pi$ is a nearest-neigibor path from $v_1$ to $v_2$\\
		using only sites in $\{v \in \Z^d:X_v=1 \}$
	\end{minipage}
	\biggr\},
\end{align*}
where $\#\pi$ is the length of a path $\pi$.
A connected component of $\{v \in \Z^d:X_v=1 \}$ which contains infinitely many points
is called an \emph{infinite cluster} for $X$.
If there exists almost surely a unique infinite cluster for $X$, then we denote it by $\mathcal{C}_\infty(X)$.

For $0<p<1$, let $\eta_p=(\eta_p(v))_{v \in \Z^d}$ denote a family of independent random variables  satisfying
\begin{align*}
	P(\eta_p(v)=1)=1-P(\eta_p(v)=0)=p,\qquad v \in \Z^d.
\end{align*}
This is called the \emph{independent Bernoulli site percolation} on $\Z^d$ of the parameter $p$.
It is well known that there is $p_c=p_c(d) \in (0,1)$ such that if $p>p_c$ then
the infinite cluster $\mathcal{C}_\infty(\eta_p)$ exists (see Theorems~1.10 and 8.1 of \cite{Gri99_book} for instance).
The following proposition presents estimates for the size of the holes in the infinite cluster $\mathcal{C}_\infty(\eta_p)$
and the chemical distance $d_{\eta_p}(\cdot,\cdot)$ (see \cite[below (2.2) and Corollary~2.2]{GarMar10} for the proof).

\begin{prop}\label{prop:GM}
For $p>p_c$, the following results (1) and (2) hold:
\begin{enumerate}
	\item There exist constants $\Cl{GM4}$ and $\Cl{GM5}$ such that for all $t>0$,
		\begin{align*}
			P( \mathcal{C}_\infty(\eta_p) \cap B_1(0,t)=\emptyset) \leq \Cr{GM4}e^{-\Cr{GM5}t}.
		\end{align*}
	\item  There exist constants $\Cl{GM1}$, $\Cl{GM2}$ and $\Cl{GM3}$ such that
		for all $v \in \Z^d$ and $t \geq \Cr{GM1}\|v\|_1$,
		\begin{align*}
			P(t \leq d_{\eta_p}(0,v)<\infty) \leq \Cr{GM2}e^{-\Cr{GM3}t}.
		\end{align*}
\end{enumerate}
\end{prop}

The proof of Theorem~\ref{thm:right} relies on the following proposition
obtained by Garet--Marchand~\cite[Theorem~1.4]{GarMar07}.
(Their argument works not only for bond percolation, but also for site percolation.)
This tells us that when $p$ is sufficiently close to one, the chemical distance looks like the $\ell^1$-norm.

\begin{prop}\label{prop:chemical}
For each $\gamma>0$, there exists $p'(\gamma) \in (p_c,1)$ such that for all $p>p'(\gamma)$,
\begin{align*}
	\limsup_{\|v\|_1 \to \infty} \frac{1}{\|v\|_1} \log P((1+\gamma)\|v\|_1 \leq d_{\eta_p}(0,v)<\infty)<0.
\end{align*}
\end{prop}

We finally recall the concept of stochastic domination.
Let $X=(X_v)_{v \in \Z^d}$ and $Y=(Y_v)_{v \in \Z^d}$
be families of random variables taking values in $\{0,1\}$.
We say that $X$ \emph{stochastically dominates} $Y$ if
\begin{align*}
	E[f(X)] \geq E[f(Y)]
\end{align*}
for all bounded, increasing, measurable functions $f:\{0,1\}^{\Z^d} \to \R$.
Furthermore, a family $X=(X_v)_{v \in \Z^d}$ of random variables is said to be \emph{finitely dependent}
if there exists $L>0$ such that any two sub-families $(X_{v_1})_{v_1 \in \Lambda_1}$
and $(X_{v_2})_{v_2 \in \Lambda_2}$ are independent whenever $\Lambda_1,\Lambda_2 \subset \Z^d$
satisfy that $\|v_1-v_2\|_1>L$ for all $v_1 \in \Lambda_1$ and $v_2 \in \Lambda_2$.

The following stochastic comparison is useful to compare locally dependent fields with the independent Bernoulli site percolation.
For the proof, we refer the reader to \cite[Theorem~7.65]{Gri99_book} or \cite[Theorem~B26]{Lig99_book} for instance.

\begin{prop}\label{prop:domi}
Suppose that $X=(X_v)_{v \in \Z^d}$ is a finitely dependent family of random variables taking values in $\{ 0,1 \}$.
For a given $0<p<1$, $X$ stochastically dominates $\eta_p$
provided $\inf_{v \in \Z^d} P(X_v=1)$ is sufficiently close to one.
\end{prop}

\subsection{Upper tail estimate for the first passage time}\label{sect:APtype}
The aim of this subsection is to prove the following proposition, which extends range of $t$
to get a bound similar to \eqref{eq:Alv_tail}.

\begin{prop}\label{prop:APtype}
There exist constants $\Cl{APtype1},\Cl{APtype2},\Cl{APtype3}$ and $\Al{APtype4}$ such that
for all $x \in \Z^d$ and $t \geq \Cr{APtype1}\|x\|_1$,
\begin{align}\label{eq:APtype}
	\P(T(0,x) \geq t) \leq \Cr{APtype2}e^{-\Cr{APtype3}t^\Ar{APtype4}}.
\end{align}
\end{prop}

Before the proof, we need some preparation.
Let $N$ be a positive integer to be chosen large enough later and set $N':=\lfloor N^{1/4}/(4d) \rfloor$.
Moreover, tile $\Z^d$ with copies of $(-N',N']^d$ such that each box is centered at a point in $\Z^d$
and each site in $\Z^d$ is contained in precisely one box.
We denote these boxes by $\Lambda_q$, $q \in \N$.
Then, a site $v$ of $\Z^d$ is said to be \emph{white} if the following conditions (1) and (2) hold:
\begin{enumerate}
	\item $\Lambda_q \cap \mathcal{I} \not= \emptyset$
		for all $q \geq 1$ with $\Lambda_q \subset B_\infty(Nv,N)$.
	\item $T(x,y) \leq N$ for all $x,y \in B_\infty(Nv,N) \cap \mathcal{I}$ with $\|x-y\|_1 \leq N^{1/4}$.
\end{enumerate}
We say that $v$ is \emph{black} otherwise.

\begin{lem}\label{lem:white}
We can find $p \in (p_c,1)$ and $N \geq 1$ such that
$(\1{\{ v \text{ is wihte} \}})_{v \in \Z^d}$ stochastically dominates $\eta_p$ and the infinite white cluster
$\mathcal{C}_\infty^\mathrm{w}:=\mathcal{C}_\infty((\1{\{ v \text{ is white} \}})_{v \in \Z^d})$ exists.
\end{lem}
\begin{proof}
Let us first check that for every $v \in \Z^d$ the event $\{ v \text{ is white} \}$ depends only on states in $B_1(Nv,2N)$.
It suffices to show that for all $x,y \in B_1(Nv,N)$, the event $\{ T(x,y) \leq N \}$ depends only on states in $B_1(Nv,2N)$.
By the definition of the first passage time, the event $\{ T(x,y) \leq N \}$ can be replaced with the event that
there exist $m \geq 1$ and $x_0,x_1,\dots,x_m \in \Z^d$ with $x_0=x$ and $x_m=y$ such that
\begin{align*}
	\sum_{i=0}^{m-1} \tau(x_i,x_{i+1}) \leq N.
\end{align*}
Since every frog can only move to an adjacent site at each step,
the above sum is strictly bigger than $N$ provided $\| x_i-x_0 \|_1>N$ for some $1 \leq i \leq m$.
Hence, $x_i$'s must satisfy $\|x_i-Nv\|_1 \leq 2N$.
This means that the event $\{ T(x,y) \leq N \}$ depends only on states in $B_1(Nv,2N)$.

We next show that $\inf_{v \in \Z^d} P(v \text{ is white})$ converges to one as $N \to \infty$.
The union bound proves
\begin{align*}
	P(0 \text{ is black})
	\leq \sum_{\substack{q \geq 1\\ \Lambda_q \subset B_\infty(0,N)}}
		P(\Lambda_q \cap \mathcal{I}=\emptyset)
	\quad 	+\sum_{\substack{x,y \in B_\infty(0,N)\\ \|x-y\|_1 \leq N^{1/4}}} \P(T(0,y-x)>N).
\end{align*}
The first summation is not larger than $cN^dP(0 \not\in \mathcal{I})^{c'N^{d/4}}$
for some constants $c$ and $c'$, and it clearly goes to zero as $N \to \infty$.
By \eqref{eq:Alv_tail}, we can also see that the second summation vanishes as $N \to \infty$.
Therefore, from translation invariance, $\inf_{v \in \Z^d} P(v \text{ is white})$ converges to one as $N \to \infty$.

With these observations, the proof is complete by using Proposition~\ref{prop:domi} and
the same strategy taken in the proof of Proposition~5.2 of \cite{Mou12}.
\end{proof}

After the preparation above, we move to the proof of Proposition~\ref{prop:APtype}.

\begin{proof}[\bf Proof of Proposition~\ref{prop:APtype}]
Without loss of generality, we can assume $\|x\|_1 \geq d^4$.
Let $p$ and $N$ be the constants appearing in Lemma~\ref{lem:white}.
Consider the events
\begin{align*}
	&\Gamma_1:=\biggl\{
		\begin{minipage}{8.5truecm}
			$\exists v_1 \in \mathcal{C}_\infty^\mathrm{w} \cap B_1(0,t^{1/4})$,
			$\exists v_2 \in \mathcal{C}_\infty^\mathrm{w} \cap B_1(v(x),t^{1/4})$\\
			such that $d^\mathrm{w}(v_1,v_2)<4\Cr{GM1}t$
		\end{minipage} \biggr\},\\
	&\Gamma_2:=\biggl\{
		\begin{minipage}{8.8truecm}
			$T(0,y)<(3N)^4t$ and $T(z,x)<(3N)^4t$ for all\\
			$y \in B_1(0,2Nt^{1/4})$ and $z \in B_1(Nv(x),2Nt^{1/4}) \cap \mathcal{I}$
		\end{minipage} \biggr\},
\end{align*}
where $d^\mathrm{w}(\cdot,\cdot)$ is the chemical distance for $(\1{\{ v \text{ is white} \}})_{v \in \Z^d}$
and $v(x)$ is the site $v$ of $\Z^d$ minimizing $\|Nv-x\|_\infty$ with a deterministic rule to break ties.
Note that on the event $\Gamma_1 \cap \Gamma_2 \cap \{ 0 \in \mathcal{I} \}$,
\begin{align*}
	T(0,x)<\{ 2(3N)^4+4\Cr{GM1}N^2 \}t,\qquad t \geq \|x\|_1.
\end{align*}

To complete the proof, we shall estimate $\P(\Gamma_1^\complement)$ and $\P(\Gamma_2^\complement)$.
Lemma~\ref{lem:white} implies that $P(\Gamma_1^\complement)$ is bounded from above by
\begin{align}\label{eq:macro}
\begin{split}
	&P \biggl(
		\begin{minipage}{9truecm}
			$d_{\eta_p}(v_1,v_2) \geq 4\Cr{GM1}t$
			for all $v_1 \in \mathcal{C}_\infty(\eta_p) \cap B_1(0,t^{1/4})$\\
			and $v_2 \in \mathcal{C}_\infty(\eta_p) \cap B_1(v(x),t^{1/4})$ 
		\end{minipage} \biggr)\\
	&\leq 2P\bigl( \mathcal{C}_\infty(\eta_p) \cap B_1(0,t^{1/4})=\emptyset \bigr)
		+\sum_{\substack{v_1 \in B_1(0,t^{1/4})\\ v_2 \in B_1(v(x),t^{1/4})}}P(4\Cr{GM1}t \leq d_{\eta_p}(v_1,v_2)<\infty).
\end{split}
\end{align}
From (1) of Proposition~\ref{prop:GM}, the first term of the right side in \eqref{eq:macro} is not larger than
$2\Cr{GM4}e^{-\Cr{GM5}t^{1/4}}$.
Note that for $t \geq \|x\|_1$, $v_1 \in B_1(0,t^{1/4})$ and $v_2 \in B_1(v(x),t^{1/4})$,
\begin{align*}
	\|v_1-v_2\|_1
	\leq 2t^{1/4}+\frac{1}{N}\|Nv(x)-x\|_1+\frac{\|x\|_1}{N}
	\leq 4t.
\end{align*}
This combined with (2) of Proposition~\ref{prop:GM} shows that the second term of the right side in \eqref{eq:macro}
is exponentially small in $t$.
Consequently, $\P(\Gamma_1^\complement)$ decays faster than $e^{-\Cr{GM5}t^{1/4}}$.
On the other hand, one has for $t \geq \|x\|_1$ and $z \in B_1(Nv(x),2Nt^{1/4})$,
\begin{align*}
	\|x-z\|_1 \leq \|x-Nv(x)\|_1+\|Nv(x)-z\|_1 \leq 3Nt^{1/4}.
\end{align*}
This together with \eqref{eq:Alv_tail} proves that $\P(\Gamma_2^\complement)$ is bounded from above by
a multiple of $t^{d/2} \exp\{ -\Cr{Alv2}(3N)^{4\Ar{Alv3}}t^\Ar{Alv3} \}$.
Therefore, \eqref{eq:APtype} immediately follows from the above bounds for
$\P(\Gamma_1^\complement)$ and $\P(\Gamma_2^\complement)$.
\end{proof}

We close this section with the corollary of Proposition~\ref{prop:APtype}.

\begin{cor}\label{cor:jump}
Suppose that $E[\omega(0)]<\infty$.
Then, there exist constants $\Cl{jump1}$, $\Cl{jump2}$ and $\Al{jump3}$ such that for all $x \in \Z^d$ and $t>0$,
\begin{align}\label{eq:neiest}
\begin{split}
	\P \biggl(
	\begin{minipage}{7.2truecm}
		$\exists v_1,v_2 \in \mathcal{I}$ with $\|v_1-v_2\|_1 \geq t$ such that\\
		$T(0,x)=T(0,v_1)+\tau(v_1,v_2)+T(v_2,x)$
	\end{minipage} \biggr)
	\leq \Cr{jump1} \|x\|_1^{2d}e^{-\Cr{jump2}t^\Ar{jump3}}.
\end{split}
\end{align}
\end{cor}
\begin{proof}
Since the left side of \eqref{eq:neiest} is smaller than or equal to $\P(T(0,x) \geq t)$,
the corollary immediately follows from Proposition~\ref{prop:APtype} provided $t \geq \Cr{APtype1}\|x\|_1$.

Assume $t<\Cr{APtype1}\|x\|_1$.
We use Proposition~\ref{prop:APtype} to obtain that the left side of \eqref{eq:neiest} is bounded from above by
\begin{align}\label{eq:neiexp}
\begin{split}
	&\Cr{APtype2}\exp\{ -\Cr{APtype3}(\Cr{APtype1}\|x\|_1)^\Ar{APtype4} \}\\
	&+\sum_{\substack{v_1,v_2 \in B_1(0,\Cr{APtype1}\|x\|_1)\\ \|v_1-v_2\|_1 \geq t}}
		\P(\tau(0,v_2-v_1)=T(0,v_2-v_1))\\
	&\leq \Cr{APtype2}e^{-\Cr{APtype3}t^\Ar{APtype4}}
		+\sum_{\substack{v_1,v_2 \in B_1(0,\Cr{APtype1}\|x\|_1)\\ \|v_1-v_2\|_1 \geq t}} \{ I_1(v_2-v_1)+I_2(v_2-v_1) \},
\end{split}
\end{align}
where for $z \in \Z^d$,
\begin{align*}
	&I_1(z):=\P \Biggl( \max_{\substack{0 \leq k \leq \Cr{APtype1}\|z\|_1\\1 \leq \ell \leq \omega(0)}}
		\|S_k(0,\ell)\|_1 \geq \|z\|_1 \Biggr),\\
	&I_2(z):=\P \Biggl( \max_{\substack{0 \leq k \leq \Cr{APtype1}\|z\|_1\\1 \leq \ell \leq \omega(0)}}
		\|S_k(0,\ell)\|_1<\|z\|_1,\,\tau(0,z)=T(0,z) \Biggr).
\end{align*}
To estimate $I_1(v_2-v_1)$, we rely on the following simple large deviation estimate for the simple random walk,
see \cite[Lemma~1.5.1]{Law91_book}:
For any $\gamma>0$, there exists a constant $c$ (which may depend on $\gamma$) such that for all $n,u \geq 0$,
\begin{align*}
	P\Bigl( \max_{0 \leq k \leq n} \|S_k(0,1)\|_1 \geq \gamma u\sqrt{n} \Bigr)
	\leq ce^{-u}.
\end{align*}
Fix $v_1,v_2 \in B_1(0,\Cr{APtype1}\|x\|_1)$ with $\|v_1-v_2\|_1 \geq t$
and set $\gamma=\Cr{APtype1}^{-1/2}$, $n=\Cr{APtype1}\|v_2-v_1\|_1$ and $u=\|v_2-v_1\|_1^{1/2}$.
Then,
\begin{align*}
	I_1(v_2-v_1)
	&\leq \sum_{L=1}^\infty \P(\omega(0)=L)
		\sum_{\ell=1}^L P\Bigl( \max_{0 \leq k \leq \Cr{APtype1}\|v_2-v_1\|_1} \|S_k(0,\ell)\|_1 \geq \|v_2-v_1\|_1 \Bigr)\\
	&\leq \E[\omega(0)] ce^{-t^{1/2}}.
\end{align*}
We use Proposition~\ref{prop:APtype} again to obtain for $v_1,v_2 \in B_1(0,\Cr{APtype1}\|x\|_1)$ with $\|v_1-v_2\|_1 \geq t$,
\begin{align*}
	I_2(v_2-v_1)
	&\leq \P(\Cr{APtype1}\|v_2-v_1\|_1<\tau(0,v_2-v_1)=T(0,v_2-v_1))\\
	&\leq \Cr{APtype2} \exp\{ -\Cr{APtype3}(\Cr{APtype1}t)^\Ar{APtype4} \}.
\end{align*}
Therefore, \eqref{eq:neiest} follows from \eqref{eq:neiexp} and these bounds for $I_1(v_2-v_1)$ and $I_2(v_2-v_1)$.
\end{proof}

\section{Right tail large deviation bound}\label{sect:right}
This section gives the proof of Theorem~\ref{thm:right}.
We basically follow the approach taken in \cite[Subsection~3.3]{GarMar07}.
Let us first prepare some notation and lemmata.

\begin{lem}\label{lem:modify}
For each $x \in \Z^d$, $P \hyphen \as$ and in $L^1$,
\begin{align}\label{eq:modify}
	\mu(x)
	=\lim_{k \to \infty} \frac{1}{k}T^*(0,kx)
	= \lim_{k \to \infty} \frac{1}{k}E[T^*(0,kx)]
	= \inf_{k \geq 1} \frac{1}{k}E[T^*(0,kx)].
\end{align}
\end{lem}
\begin{proof}
From Proposition~\ref{prop:Alv_shape}, we have on the event $\{ 0 \in \mathcal{I} \}$ of positive probability,
\begin{align*}
	\mu(x)
	= \lim_{\substack{k \to \infty\\ kx \in \mathcal{I}}} \frac{1}{k}T(0,kx)
	= \lim_{\substack{k \to \infty\\ kx \in \mathcal{I}}} \frac{1}{k}T^*(0,kx).
\end{align*}
Therefore, once the integrability of $T^*(0,x)$ is proved, \eqref{eq:modify} follows from the subadditive ergodic theorem
for the process $T^*(ix,jx)$, $0 \leq i<j$, $i,j \in \N_0$.

For the integrability,
\begin{align*}
\begin{split}
	E[T^*(0,x)]
	&\leq \int_0^\infty P \Bigl (\|0^*\|_1>\frac{t}{3\Cr{APtype1}} \Bigr)\,dt
		+\int_0^\infty P \Bigl( \|x-x^*\|_1>\frac{t}{3\Cr{APtype1}} \Bigr)\,dt\\
	&\quad	+\int_0^\infty P \Bigl( T^*(0,x) \geq t,\,\|0^*\|_1 \leq \frac{t}{3\Cr{APtype1}},
		\,\|x-x^*\|_1 \leq \frac{t}{3\Cr{APtype1}} \Bigr)\,dt.
\end{split}
\end{align*}
It is clear that the first and second terms in the right side are finite.
Moreover, the third term is not larger than
\begin{align*}
	3\Cr{APtype1}\|x\|_1
	+\sum_{\substack{y \in B_1(0,t/(3\Cr{APtype1}))\\ z \in B_1(x,t/(3\Cr{APtype1}))}}
	\int_{3\Cr{APtype1}\|x\|_1}^\infty \P(T(0,z-y) \geq t)\,dt,
\end{align*}
and the integrability of $T^*(0,x)$ follows by using Proposition~\ref{prop:APtype}.
\end{proof}

We denote by $\mathcal{S}_d$ the symmetric group on $\{1,\dots,d\}$.
For each $x=(x_1,\dots,x_d) \in \R^d$, $\sigma \in \mathcal{S}_d$ and $\epsilon \in \{ +1,-1 \}^d$, we define
\begin{align*}
	\Psi_{\sigma,\epsilon}(x):=\sum_{i=1}^d \epsilon(i)x_{\sigma(i)}\xi_i.
\end{align*}
Then, $\mathcal{O}(\Z^d):=\{ \Psi_{\sigma,\epsilon}: \sigma \in \mathcal{S}_d,\,\epsilon \in \{+1,-1\}^d \}$
is the group of orthogonal transformations that preserve the grid $\Z^d$.
Consequently, its elements also preserve the $\ell^1$-norm $\|\cdot\|_1$ and the time constant $\mu(\cdot)$.
For $x \in \R^d$ and $(g_1,\dots,g_d) \in (\mathcal{O}(\Z^d))^d$,
the linear map $L_x^{g_1,\dots,g_d}$ is defined by
\begin{align*}
	L_x^{g_1,\dots,g_d}(y):=\sum_{i=1}^dy_ig_i(x),\qquad y=(y_1,\dots,y_d) \in \R^d.
\end{align*}

To study the first passage time in each direction $x$, we want to find a basis of $\R^d$ adapted to the studied direction,
i.e., made of images of $x$ by elements of $\mathcal{O}(\Z^d)$.
The following technical lemma, which is obtained by Garet--Marchand~\cite[Lemma~2.2]{GarMar07}, gives the existence of such a basis.

\begin{lem}\label{lem:sym}
For each $x \in \R^d$, there exists a family $(g_{1,x},g_{2,x},\dots,g_{d,x}) \in (\mathcal{O}(\Z^d))^d$
with $g_{1,x}=\Id_{\R^d}$ such that the linear map $L_x:=L_x^{g_{1,x},\dots,g_{d,x}}$ satisfies
\begin{align*}
	\Cr{sym1} \|x\|_1\|y\|_1 \leq \|L_x(y)\|_1 \leq \|x\|_1\|y\|_1,\qquad y \in \R^d,
\end{align*}
where $\Cl{sym1}$ is a universal constant not depending on $x$, $y$ and $(g_{1,x},g_{2,x},\dots,g_{d,x})$.
\end{lem}

\begin{proof}[\bf Proof of Theorem~\ref{thm:right}]
We fix an arbitrary $\epsilon>0$ and break the proof into three steps.

\paragraph{\bf Step~1}
In this step, we choose appropriate constants for our proof.
By \eqref{eq:tc_bound}, $\mu(y) \geq 1$ holds for all $y \in \R^d$ with $\|y\|_1=1$.
Hence, there exists $\delta>0$ such that for all $y \in \R^d$ with $\|y\|_1=1$,
\begin{align}\label{eq:direction}
	\biggl( 1+\frac{3\delta}{2\Cr{APtype1}} \biggr)(1+\delta)^2\mu(y)+2\delta<\mu(y)(1+\epsilon)
\end{align}
and
\begin{align*}
	\delta<\frac{\Cr{APtype1}}{2}.
\end{align*}
To shorten notation, write
\begin{align*}
	\beta:=\frac{\delta}{2\Cr{APtype1}}<\frac{1}{4}.
\end{align*}
Take $M \in \N$ large enough to have
\begin{align}\label{eq:choiceM}
	M \geq \frac{d}{\delta} \max\biggl\{ \frac{\mu(\xi_1)}{2},\frac{8\Cr{APtype1}}{\Cr{sym1}} \biggr\} \geq 4,
\end{align}
and choose $p \in (0,1)$ to satisfy
\begin{align}\label{eq:p}
	p>p'\biggl( \frac{\beta}{1+2\beta} \biggr)>p_c,
\end{align}
where $p'(\cdot)$ is the parameter appearing in Proposition~\ref{prop:chemical}.

\paragraph{\bf Step~2}
In this step, we tackle the construction of the renormalization procedure.
Let $N$ be a positive integer to be chosen large enough later.
A site $v \in \Z^d$ is said to be \emph{good} if the following conditions (1) and  (2) hold
for all $y \in \frac{1}{M}\Z^d$ with $\|y\|_1=1$:
\begin{enumerate}
	\item $T^*(NL_{My}(v),NL_{My}(v+\xi)) \leq MN\mu(y)(1+\delta)$ for all $\xi \in \mathcal{E}^d$.
	\item $(NL_{My}(v))^*$ is included in $B_1(NL_{My}(v),\sqrt{N})$, and
		$(NL_{My}(v+\xi))^*$ belongs to $B_1(NL_{My}(v+\xi),\sqrt{N})$ for all $\xi \in \mathcal{E}^d$.
\end{enumerate}
Otherwise, $v$ is called \emph{bad}.

\begin{lem}\label{lem:good}
There exists $N \in \N$ such that $(\1{\{ v \text{ is good} \}})_{v \in \Z^d}$ stochastically dominates $\eta_p$.
\end{lem}
\begin{proof}
Since the set $\{ y \in \frac{1}{M}\Z^d: \|y\|_1=1 \}$ is finite,
Lemma~\ref{lem:modify} implies that $P(v \text{ is bad})$ converges to zero as $N \to \infty$.
Moreover, Lemma~\ref{lem:sym} ensures that if $\|v-w\|_1>(2/\Cr{sym1})\mu(\xi_1)(1+\delta)$, then
for all $y \in \frac{1}{M}\Z^d$ with $\|y\|_1=1$,
\begin{align*}
	\| NL_{My}(v)-NL_{My}(w)\|_1>2MN\mu(\xi_1)(1+\delta).
\end{align*}
This means that $(\1{\{ v \text{ is good} \}})_{v \in \Z^d}$ is finitely dependent.
Therefore, the lemma follows from Proposition~\ref{prop:domi}.
\end{proof}


For a given $x \in \Z^d \setminus \{0\}$, we set $x':=x/\|x\|_1$.
Then, there exists $\hat{x} \in \frac{1}{M}\Z^d$ such that $\|\hat{x}\|_1=1$ and $\| x'-\hat{x} \|_1 \leq d/(2M)$.
Note that, by \eqref{eq:tc_bound} and \eqref{eq:choiceM},
\begin{align}\label{eq:mudash}
	|\mu(x')-\mu(\hat{x})|
	\leq \mu(\xi_1) \|x'-\hat{x}\|_1
	\leq \mu(\xi_1) \,\frac{d}{2M}
	\leq \delta
\end{align}
and
\begin{align}\label{eq:dash}
	\| x'-\hat{x} \|_1 \leq \frac{\Cr{sym1} \beta}{8}.
\end{align}
The definition of $L_{M\hat{x}}$ and Lemma~\ref{lem:sym} tell us that for all $1 \leq i \leq d$,
\begin{align*}
	\mu(L_{M\hat{x}}(\xi_i))=M\mu(\hat{x}),\qquad
	\|L_{M\hat{x}}(\xi_i)\|_1=M,
\end{align*}
and for all $y \in \R^d$,
\begin{align*}
	\Cr{sym1}M\|y\|_1 \leq \|L_{M\hat{x}}(y)\|_1 \leq M\|y\|_1.
\end{align*}

Denote by $d^\mathrm{g}(\cdot,\cdot)$ the chemical distance for $(\1{\{ v \text{ is good} \}})_{v \in \Z^d}$.
We now consider the event
\begin{align*}
	G:=\bigl\{ \exists v \in \mathcal{A}(0,\beta \|\bar{x}\|_1),
	\,\exists w \in \mathcal{A}(\bar{x},\beta \|\bar{x}\|_1)
	\text{ suth that } d^\mathrm{g}(v,w)<(1+3\beta)\|\bar{x}\|_1 \bigr\},
\end{align*}
where
\begin{align*}
	\bar{x}:=\biggl\lfloor \frac{\|x\|_1}{MN} \biggr\rfloor \xi_1
\end{align*}
and 
\begin{align*}
	\mathcal{A}(z,r):=\{ y \in \Z^d: r/2 \leq \|y-z\|_1 \leq r \},\qquad z \in \Z^d,\,r>0.
\end{align*}
It is easy to see that on the event $G$, for some $v \in \mathcal{A}(0,\beta \|\bar{x}\|_1)$
and $w \in \mathcal{A}(\bar{x},\beta \|\bar{x}\|_1)$,
\begin{align}\label{eq:gpath}
	T^*(NL_{M\hat{x}}(v),NL_{M\hat{x}}(w))
	<MN\mu(\hat{x})(1+\delta)(1+3\beta)\|\bar{x}\|_1.
\end{align}
Furthermore, Lemma~\ref{lem:good} proves
\begin{align*}
	P(G^\complement)
	&\leq P\biggl(
		\begin{minipage}{6.8truecm}
			$d_{\eta_p}(v,w) \geq (1+3\beta)\|\bar{x}\|_1$ for all\\
			$v \in \mathcal{A}(0,\beta \|\bar{x}\|_1)$ and $w \in \mathcal{A}(\bar{x},\beta \|\bar{x}\|_1)$
		\end{minipage}\biggr)\\
	&\leq 2P\bigl( B_1(0,\beta\|\bar{x}\|_1/2) \cap \mathcal{C}_\infty(\eta_p)=\emptyset \bigr)\\
	&\quad +\sum_{\substack{v \in B_1(0,\beta \|\bar{x}\|_1)\\w \in B_1(\bar{x},\beta \|\bar{x}\|_1)}}
		P\biggl( \frac{1+3\beta}{1+2\beta}\|v-w\|_1 \leq d_{\eta_p}(v,w)<\infty \biggr).
\end{align*}
Thanks to $\beta<1/4$ and \eqref{eq:p}, (1) of Proposition~\ref{prop:GM} and Proposition~\ref{prop:chemical}
imply that for some constants $c$ and $c'$,
\begin{align}\label{eq:Gc}
	P(G^c)
	\leq ce^{-\Cr{GM5}\beta\|x\|_1/(2MN)}
		+c\biggl( \frac{\beta\|x\|_1}{MN} \biggr)^{2d}e^{-c'(1-2\beta)\|x\|_1/(MN)}.
\end{align}

\paragraph{\bf Step~3}
Finally, we complete the proof.
There is no loss of generality in assuming
\begin{align}\label{eq:xchoice}
	\|x\|_1 \geq \frac{4MN}{\beta \Cr{sym1}}.
\end{align}
By the definition of $x'$ and \eqref{eq:direction},
\begin{align}\label{eq:final}
\begin{split}
	&\P(T(0,x) \geq (1+\epsilon)\mu(x))\\
	&= \P(T(0,x) \geq \mu(x')(1+\epsilon)\|x\|_1)\\
	&\leq \P\biggl( T(0,x)
		>\biggl( 1+\frac{3\delta}{2\Cr{APtype1}} \biggr)(1+\delta)^2 \mu(x')\|x\|_1+2\delta\|x\|_1 \biggr).
\end{split}
\end{align}
Let $A$ be the event that $T(0,y)<\delta \|x\|_1$ for all $y \in NL_{M\hat{x}}(\mathcal{A}(0,\beta\|\bar{x}\|_1))+B_1(0,\sqrt{N})$
and $T(z,x)<\delta \|x\|_1$ for all
$z \in [NL_{M\hat{x}}(\mathcal{A}(\bar{x},\beta\|\bar{x}\|_1))+B_1(0,\sqrt{N})] \cap \mathcal{I}$.
By \eqref{eq:mudash} and \eqref{eq:gpath}, on the event $G \cap A \cap \{ 0 \in \mathcal{I} \}$,
there exist $v \in \mathcal{A}(0,\beta\|\bar{x}\|_1)$ and $w \in \mathcal{A}(\bar{x},\beta\|\bar{x}\|_1)$ such that
\begin{align*}
	T(0,x)
	&\leq T(0,(NL_{M\hat{x}}(v))^*)
		+T^*(NL_{M\hat{x}}(v),NL_{M\hat{x}}(w))
		+T((NL_{M\hat{x}}(w))^*,x)\\
	&\leq MN \mu(\hat{x})(1+\delta)(1+3\beta)\|\bar{x}\|_1+2\delta\|x\|_1\\
	&\leq \biggl( 1+\frac{3\delta}{2\Cr{APtype1}} \biggr)(1+\delta)^2 \mu(x')\|x\|_1+2\delta\|x\|_1.
\end{align*}
This means that the most right side of \eqref{eq:final} is bounded from above by $\P(G^\complement)+\P(A^\complement)$.
Due to \eqref{eq:Gc}, our task is to estimate $\P(A^\complement)$.
We use Lemma~\ref{lem:sym} and \eqref{eq:xchoice} to obtain that for
$y \in NL_{M\hat{x}}(\mathcal{A}(0,\beta\|\bar{x}\|_1))+B_1(0,\sqrt{N})$,
\begin{align*}
	\|y\|_1
	\leq MN \beta \|\bar{x}\|_1+\sqrt{N}
	\leq \beta \|x\|_1+\sqrt{N}
	\leq \frac{17}{16}\beta \|x\|_1
	\leq \frac{\delta}{\Cr{APtype1}}\|x\|_1
\end{align*}
and
\begin{align*}
	\|y\|_1
	&\geq MN \Cr{sym1} \,\frac{\beta \|\bar{x}\|_1}{2}-\sqrt{N}\\
	&\geq \Cr{sym1}\frac{\beta \|x\|_1}{2}-MN\frac{\Cr{sym1}\beta}{2}-\sqrt{N}
		\geq \frac{13}{32} \Cr{sym1}\beta\|x\|_1.
\end{align*}
Similarly, for $z \in NL_{M\hat{x}}(\mathcal{A}(\bar{x},\beta\|\bar{x}\|_1))+B_1(0,\sqrt{N})$,
\begin{align*}
	\frac{13}{32} \Cr{sym1}\beta\|x\|_1 \leq \|z-NL_{M\hat{x}}(\bar{x})\|_1 \leq \frac{17}{16} \beta \|x\|_1.
\end{align*}
In addition, by \eqref{eq:dash}, one has that for $z \in NL_{M\hat{x}}(\mathcal{A}(\bar{x},\beta\|\bar{x}\|_1))+B_1(0,\sqrt{N})$,
\begin{align*}
	\|x-z\|_1
	&\leq \|x-NL_{M\hat{x}}(\bar{x})\|_1+\|NL_{M\hat{x}}(\bar{x})-z\|_1\\
	&\leq \frac{3}{8}\Cr{sym1}\beta\|x\|_1+\frac{17}{16}\beta\|x\|_1
		\leq 2\beta \|x\|_1=\frac{\delta}{\Cr{APtype1}}\|x\|_1
\end{align*}
and
\begin{align*}
	\|x-z\|_1
	&\geq \|z-NL_{M\hat{x}}(\bar{x})\|_1-\|NL_{M\hat{x}}(\bar{x})-x\|_1\\
	&\geq \frac{13}{32}\Cr{sym1}\beta\|x\|_1-\frac{3}{8}\Cr{sym1}\beta\|x\|_1
		= \frac{\Cr{sym1}}{32}\beta\|x\|_1.
\end{align*}
Therefore,
\begin{align*}
	\P(A^\complement)
	&\leq \sum_{(13/32)\Cr{sym1}\beta\|x\|_1 \leq \|y\|_1 \leq (17/16) \beta \|x\|_1}
		\P(T(0,y)>\Cr{APtype1} \|y\|_1)\\
	&\quad +\sum_{(13/32)\Cr{sym1}\beta\|x\|_1 \leq \|z-NL_{M\hat{x}}(\bar{x})\|_1 \leq (17/16)\beta \|x\|_1}
		\P(T(0,x-z)>\Cr{APtype1}\|x-z\|_1),
\end{align*}
and this combined with Proposition~\ref{prop:APtype} completes the proof.
\end{proof}

\section{Left tail large deviation and concentration bounds}\label{sect:left}
The aim of this section is to prove Theorems~\ref{thm:left} and \ref{thm:tscon}.
Let us begin with the proof of Theorem~\ref{thm:left}.

\begin{proof}[\bf Proof of Theorem~\ref{thm:left}]
Let $v(x)$ denote a site of $\mathcal{I}$ satisfying
\begin{align*}
	T(0^*,x)=T(0^*,v(x))+\tau(v(x),x).
\end{align*}
We first prove that for all $\epsilon>0$ there exist constants $\Cl{bypath1}$ and $\Cl{bypath2}$ such that
\begin{align}\label{eq:bypath}
	P(T(v(x),x^*) \geq \epsilon\|x\|_1) \leq \Cr{bypath1} \exp\{ -\Cr{bypath2}\|x\|_1^{\Ar{APtype4} \wedge \Ar{jump3}} \}.
\end{align}
Corollary~\ref{cor:jump} tells us that there exist constants $c$ and $c'$ such that
\begin{align*}
	P \biggl(
	\begin{minipage}{8.8truecm}
		$\exists v_1,v_2 \in \mathcal{I}$ with $\|v_1-v_2\|_1 \geq\epsilon(2\Cr{APtype1})^{-1}\|x\|_1$ such\\
		that $T(0^*,x)=T(0^*,v_1)+\tau(v_1,v_2)+T(v_2,x)$
	\end{minipage} \biggr)
	\leq ce^{-c'\|x\|_1^\Ar{jump3}}.
\end{align*}
It follows that
\begin{align*}
	&P(T(v(x),x^*) \geq \epsilon\|x\|_1)\\
	&\leq ce^{-c'\|x\|_1^\Ar{jump3}}
		+P\biggl( \|x-x^*\|_1 \geq \frac{\epsilon\|x\|_1}{2\Cr{APtype1}} \biggr)\\
	&\quad +P \biggl( T(v(x),x^*) \geq \epsilon\|x\|_1,\, \|v(x)-x\|_1<\frac{\epsilon\|x\|_1}{2\Cr{APtype1}},\,
		\|x-x^*\|_1<\frac{\epsilon\|x\|_1}{2\Cr{APtype1}} \biggr).
\end{align*}
Since the second term has the desired form, our task is to bound the last probability.
To this end, we use Proposition~\ref{prop:APtype} to obtain that for some constants $C$ and $C'$,
\begin{align*}
	&P \biggl( T(v(x),x^*) \geq \epsilon\|x\|_1,\, \|v(x)-x\|_1<\frac{\epsilon\|x\|_1}{2\Cr{APtype1}},\,
		\|x-x^*\|_1<\frac{\epsilon\|x\|_1}{2\Cr{APtype1}} \biggr)\\
	&\leq \sum_{\substack{y \in \Z^d\\ \|y-x\|_1<(2\Cr{APtype1})^{-1}\epsilon\|x\|_1}}
		\sum_{\substack{z \in \Z^d\\ \|x-z\|_1<(2\Cr{APtype1})^{-1}\epsilon\|x\|_1}}
		\P(T(0,z-y) \geq \epsilon\|x\|_1)\\
	&\leq Ce^{-C'\|x\|_1^\Ar{APtype4}}.
\end{align*}
Hence, \eqref{eq:bypath} follows.

Taking $t=\epsilon \sqrt{\|x\|_1}$, one has by \eqref{eq:tc_bound} and \eqref{eq:modify},
\begin{align*}
	&\P(T(0,x) \leq (1-\epsilon)\mu(x))\\
	&\leq P( 0 \in \mathcal{I})^{-1}P\Bigl( T(0^*,x)-E[T^*(0,x)] \leq -t\sqrt{\|x\|_1} \Bigr).
\end{align*}
The last probability is bounded from above by
\begin{align*}
	P\biggl( T(0^*,x)-T^*(0,x) \leq -\frac{t}{2}\sqrt{\|x\|_1} \biggr)
	+P\biggl( T^*(0,x)-E[T^*(0,x)] \leq -\frac{t}{2}\sqrt{\|x\|_1} \biggr).
\end{align*}
Note
\begin{align*}
	T^*(0,x)
	\leq T(0^*,v(x))+T(v(x),x^*)+\tau(v(x),x)
	\leq T(0^*,x)+T(v(x),x^*),
\end{align*}
and \eqref{eq:bypath} implies
\begin{align*}
	P\biggl( T(0^*,x)-T^*(0,x) \leq -\frac{t}{2}\sqrt{\|x\|_1} \biggr)
	&\leq P\biggl( T(v(x),x^*) \geq \frac{\epsilon}{2}\|x\|_1 \biggr)\\
	&\leq \Cr{bypath1} \exp\{ -\Cr{bypath2}\|x\|_1^{\Ar{APtype4} \wedge \Ar{jump3}} \}.
\end{align*}
Furthermore, once Theorem~\ref{thm:tscon} is proved, one has
\begin{align*}
	P\biggl( T^*(0,x)-E[T^*(0,x)] \leq -\frac{t}{2}\sqrt{\|x\|_1} \biggr)
	\leq \Cr{scon2}\exp\Bigl\{ -\Cr{scon3} \Bigl( \frac{\epsilon}{2} \sqrt{\|x\|_1} \Bigr)^\Ar{scon4} \Bigr\},
\end{align*}
and therefore the theorem follows.
\end{proof}

\begin{proof}[\bf Proof of Theorem~\ref{thm:tscon}]
For each $t>0$, define the two-point function $\sigma_t(\cdot,\cdot)$ as follows.
Take $K>d(\Cr{APtype1}+\gamma+1)$.
First, if $\|x-y\|_\infty \leq t$ and $\tau(x,y)>4Kt$, then set $\sigma_t(x,y):=4Kt$.
Next, if $\|x-y\|_\infty>t$ then set $\sigma_t(x,y):=4K\|x-y\|_\infty$.
Otherwise, set $\sigma_t(x,y):=\tau(x,y)$.
By definition, for any $x,y \in \Z^d$,
\begin{align}\label{eq:ttbd}
	\|x-y\|_1 \leq \sigma_t(x,y) \leq 4K(t \vee \|x-y\|_\infty).
\end{align}
We write $T_t(x,y)$ for the first passage time from $x$ to $y$ corresponding to $\sigma_t(\cdot,\cdot)$, i.e.,
\begin{align*}
	T_t(x,y):=\inf\Biggl\{ \sum_{i=0}^{m-1}\sigma_t(x_i,x_{i+1}):
	\begin{minipage}{4truecm}
		$m \geq 1$,\\
		$x=x_0,x_1,\dots,x_m=y$
	\end{minipage}
	\Biggr\}.
\end{align*}

\begin{prop}\label{prop:tsdif}
There exist constants $\Cl{tsdif1}$, $\Cl{tsdif2}$ and $\Al{tsdif3}$ such that for all $x \in \Z^d \setminus \{0\}$
and $0 \leq t \leq \|x\|_1$,
\begin{align*}
	&\max\bigl\{ P(T_t(0^*,x^*) \not= T^*(0,x)),E\bigl[(T_t(0^*,x^*)-T^*(0,x))^2 \bigr]^{1/2} \bigr\}\\
	&\leq \Cr{tsdif1}\|x\|_1^{4d}e^{-\Cr{tsdif2}t^\Ar{tsdif3}}.
\end{align*}
\end{prop}

\begin{prop}\label{prop:ttcon}
For all $\gamma>0$ there exists a constant $\Cl{ttcon1}$ such that for all $x \in \Z^d \setminus \{0\}$ and
$0 \leq t \leq \gamma\sqrt{\|x\|_1}$,
\begin{align*}
	P\Bigl( |T_t(0,x)-E[T_t(0,x)]| \geq t\sqrt{\|x\|_1} \Bigr) \leq 2e^{-\Cr{ttcon1}t}.
\end{align*}
\end{prop}

Let us postpone the proofs of these propositions to the end of this section, and continue the proof of Theorem~\ref{thm:tscon}.
To this end, without loss of generality we can assume $\|x\|_1 \geq (32KE[1 \vee \|0^*\|_\infty])^2$.
Take $c \geq 1$ large enough to have for all $t \geq c(1+\log\|x\|_1)^{1/\Ar{tsdif3}}$,
\begin{align*}
	\Cr{tsdif1}\|x\|_1^{4d}e^{-\Cr{tsdif2}t^\Ar{tsdif3}}
	\leq \Cr{tsdif1}e^{-\Cr{tsdif2}t^\Ar{tsdif3}/2}
	\leq \frac{t}{4}.	
\end{align*}
From \eqref{eq:ttbd} and Proposition~\ref{prop:tsdif}, we have
\begin{align*}
	&|E[T^*(0,x)]-E[T_t(0,x)]|\\
	&\leq E[|T^*(0,x)-T_t(0^*,x^*)|]+2E[T_t(0,0^*) \vee T_t(0^*,0)]\\
	&\leq \Cr{tsdif1}\|x\|_1^{4d}e^{-\Cr{tsdif2}t^\Ar{tsdif3}}+8KE[t \vee \|0^*\|_\infty].
\end{align*}
Hence, for all $t \geq c(1+\log\|x\|_1)^{1/\Ar{tsdif3}}$,
\begin{align*}
	|E[T^*(0,x)]-E[T_t(0,x)]| \leq \frac{t}{2}\sqrt{\|x\|_1}.
\end{align*}
This together with Proposition~\ref{prop:tsdif} leads to
\begin{align*}
	&P\Bigl( |T^*(0,x)-E[T^*(0,x)]| \geq t\sqrt{\|x\|_1} \Bigr)\\
	&\leq \Cr{tsdif1}e^{-\Cr{tsdif2}t^\Ar{tsdif3}/2}
		+P\biggl( |T_t(0^*,x^*)-E[T_t(0,x)]| \geq \frac{t}{2}\sqrt{\|x\|_1} \biggr).
\end{align*}
For the second term of the right side,
\begin{align*}
	&|T_t(0^*,x^*)-E[T_t(0,x)]|\\
	&\leq |T_t(0^*,x^*)-T_t(0,x)|+|T_t(0,x)-E[T_t(0,x)]|\\
	&\leq T_t(0,0^*) \vee T_t(0^*,0)+T_t(x,x^*) \vee T_t(x^*,x)+|T_t(0,x)-E[T_t(0,x)]|.
\end{align*}
We use \eqref{eq:ttbd} again to obtain that
\begin{align*}
	&P\biggl( |T_t(0^*,x^*)-E[T_t(0,x)]| \geq \frac{t}{2}\sqrt{\|x\|_1} \biggr)\\
	&\leq 2P\biggl( 4K(t \vee \|0^*\|_\infty) \geq \frac{t}{6}\sqrt{\|x\|_1} \biggr)
		+P\biggl( |T_t(0,x)-E[T_t(0,x)]| \geq \frac{t}{6}\sqrt{\|x\|_1} \biggr),
\end{align*}
and the theorem is a consequence of Proposition~\ref{prop:ttcon}.
\end{proof}

In the rest of this section, we shall prove Propositions~\ref{prop:tsdif} and \ref{prop:ttcon}.

\begin{proof}[\bf Proof of Proposition~\ref{prop:tsdif}]
Let $0 \leq t \leq \|x\|_1$.
We first estimate $P(T_t(0^*,x^*) \not= T^*(0,x))$.
To this end, consider the following events $\Gamma_j$, $1 \leq j \leq 5$:
\begin{align*}
	&\Gamma_1:=\{ \|0^*\|_1 \leq t/8,\,\|x-x^*\|_1 \leq t/8 \},\\
	&\Gamma_2:=\bigcap_{\substack{y,z \in B_1(0,7K\|x\|_1)}}
		\biggl\{
		y \in \mathcal{I} \Longrightarrow
		\begin{minipage}{7truecm}
			$T(y,z) \not= T(y,v_1)+\tau(v_1,v_2)+T(v_2,z)$\\
			for all $v_1,v_2 \in \mathcal{I}$ with $\|v_1-v_2\|_1 \geq t/2$
		\end{minipage} \biggr\},\\
	&\Gamma_3:=\bigcap_{\substack{y,z \in B_1(0,7K\|x\|_1)\\ \|y-z\|_\infty \leq 2t}}
		\{ y \in \mathcal{I} \Longrightarrow T(y,z) \leq 2Kt \},\\
	&\Gamma_4:=\bigcap_{\substack{y,z \in B_1(0,7K\|x\|_1)\\ \|y-z\|_\infty \geq t/2}}
		\{ y \in \mathcal{I} \Longrightarrow T(y,z) \leq K\|y-z\|_\infty \},\\
	&\Gamma_5:=\bigcap_{y,z \in B_1(0,7K\|x\|_1)} \{ y \in \mathcal{I} \Longrightarrow T(y,z) \geq T_t(y,z) \}.
\end{align*}

We shall observe that $T_t(0^*,x^*)=T^*(0,x)$ holds on the event $\bigcap_{j=1}^5 \Gamma_j$.
Denote by $(x_i)_{i=0}^m$ a finite sequence of $\Z^d$ satisfying that $x_0=0^*$, $x_m=x^*$ and
$T_t(0^*,x^*)=\sum_{i=0}^{m-1}\sigma_t(x_i,x_{i+1})$.
Moreover, the index $i_0$ is defined by
\begin{align*}
	i_0:=\max\{0 \leq i \leq m:T(0^*,x_i)=T_t(0^*,x_i) \}.
\end{align*}
On the event $\Gamma_1$, we have $\|0^*\|_1 \leq K\|x\|_1$ and it holds by \eqref{eq:ttbd} that
\begin{align*}
	T_t(0^*,x^*)
	\leq 4K(t \vee \|0^*-x^*\|_\infty)
	\leq 4K \biggl\{ t \vee \biggl( \frac{t}{4}+\|x\|_\infty \biggr) \biggr\}
	\leq 5K\|x\|_1.
\end{align*}
This proves that $x_i$'s are included in $B_1(0,6K\|x\|_1)$ on the event $\Gamma_1$.
Let $x'_{i_0}$ denote a site of $\mathcal{I}$ satisfying $T(0^*,x_{i_0})=T(0^*,x'_{i_0})+\tau(x'_{i_0},x_{i_0})$.
Note that $\| x'_{i_0}-x_{i_0}\|_1 \leq t/2$ and $\|x'_{i_0}\|_1 \leq 7K\|x\|_1$ on the event $\Gamma_1 \cap \Gamma_2$.
Assume $i_0<m$.
Then, on the event $\Gamma_1 \cap \Gamma_5$,
\begin{align*}
	T(0^*,x_{i_0+1})
	&>T_t(0^*,x_{i_0+1})=T_t(0^*,x_{i_0})+\sigma_t(x_{i_0},x_{i_0+1})\\
	&=T(0^*,x_{i_0})+\sigma_t(x_{i_0},x_{i_0+1}).
\end{align*}
We now consider the following three cases:
\begin{enumerate}
	\item $\|x_{i_0}-x_{i_0+1}\|_\infty \leq t$ and $\tau(x_{i_0},x_{i_0+1})>4Kt$,
	\item $\|x_{i_0}-x_{i_0+1}\|_\infty>t$,
	\item $\|x_{i_0}-x_{i_0+1}\|_\infty \leq t$ and $\tau(x_{i_0},x_{i_0+1}) \leq 4Kt$.
\end{enumerate}

\paragraph{\bf Case~(1)}
On the event $\Gamma_1 \cap \Gamma_2$,
\begin{align*}
	\|x'_{i_0}-x_{i_0+1}\|_\infty
	\leq \| x'_{i_0}-x_{i_0}\|_\infty+\|x_{i_0}-x_{i_0+1}\|_\infty
	\leq \frac{t}{2}+t \leq 2t.
\end{align*}
Therefore, on the event $\Gamma_1 \cap \Gamma_2 \cap \Gamma_3 \cap \Gamma_5$,
\begin{align*}
	T(0^*,x_{i_0+1})
	&>T(0^*,x_{i_0})+\sigma_t(x_{i_0},x_{i_0+1})=T(0^*,x_{i_0})+4Kt\\
	&\geq T(0^*,x'_{i_0})+T(x'_{i_0},x_{i_0+1}) \geq T(0^*,x_{i_0+1}).
\end{align*}
This is a contradiction.

\paragraph{\bf Case~(2)}
On the event $\Gamma_1 \cap \Gamma_2$,
\begin{align*}
	\|x'_{i_0}-x_{i_0+1}\|_\infty \geq \|x_{i_0+1}-x_{i_0}\|_\infty-\|x_{i_0}-x'_{i_0}\|_\infty
	\geq \frac{t}{2},
\end{align*}
and on the event $\Gamma_1 \cap \Gamma_2 \cap \Gamma_5$,
\begin{align*}
	T(0^*,x_{i_0+1})
	&> T(0^*,x_{i_0})+\sigma_t(x_{i_0},x_{i_0+1})=T(0^*,x_{i_0})+4K\|x_{i_0}-x_{i_0+1}\|_\infty\\
	&\geq T(0^*,x'_{i_0})+2Kt+2K\|x_{i_0}-x_{i_0+1}\|_\infty.
\end{align*}
It follows that on the event $\Gamma_1 \cap \Gamma_2 \cap \Gamma_4 \cap \Gamma_5$,
\begin{align*}
	T(0^*,x_{i_0+1})
	&> T(0^*,x'_{i_0})+2K\|x'_{i_0}-x_{i_0}\|_\infty+2K\|x_{i_0}-x_{i_0+1}\|_\infty\\
	&\geq T(0^*,x'_{i_0})+2K\|x'_{i_0}-x_{i_0+1}\|_\infty\\
	&\geq T(0^*,x'_{i_0})+T(x'_{i_0},x_{i_0+1}) \geq T(0^*,x_{i_0+1}),
\end{align*}
and this leads to another contradiction.

\paragraph{\bf Case~(3)}
Since $\sigma_t(x_{i_0},x_{i_0+1})=\tau(x_{i_0},x_{i_0+1})$, on the event $\Gamma_1 \cap \Gamma_5$,
\begin{align*}
	T(0^*,x_{i_0+1})
	&> T(0^*,x_{i_0})+\sigma_t(x_{i_0},x_{i_0+1})\\
	&= T(0^*,x_{i_0})+\tau(x_{i_0},x_{i_0+1}) \geq T(0^*,x_{i_0+1}),
\end{align*}
which is also a contradiction.

With these observations, on the event $\bigcap_{j=1}^5 \Gamma_j$,
$i_0=m$ must hold and $T_t(0^*,x^*)=T^*(0,x)$ is valid.
It remains to estimate the probability of $\bigcup_{j=1}^5 \Gamma_j^\complement$.
Obviously, $P(\Gamma_1^\complement)$ is exponentially small in $t$.
The following bound is an immediate consequence of Proposition~\ref{prop:APtype} and Corollary~\ref{cor:jump}:
For some constants $c$ and $c'$,
\begin{align*}
	P(\Gamma_2^\complement)+P(\Gamma_3^\complement)+P(\Gamma_4^\complement)
	\leq c\|x\|_1^{4d}\exp\{-c't^{\Ar{APtype4} \wedge \Ar{jump3}}\}.
\end{align*}
To estimate $P(\Gamma_5^\complement)$, let us introduce the event $\Gamma_6(w)$ that
$T(0,w) \not= T(0,v_1)+\tau(v_1,v_2)+T(v_2,w)$ for all $v_1,v_2 \in \mathcal{I}$ with $\|v_1-v_2\|_\infty \geq t$.
Since $T(0,w) \geq T_t(0,w)$ on the event $\Gamma_6(w) \cap \{ 0 \in \mathcal{I} \}$, one has
\begin{align*}
	P(\Gamma_5^\complement)
	&\leq \sum_{y,z \in B_1(0,7K\|x\|_1)}\P(T(0,z-y)-T_t(0,z-y)<0)\\
	&\leq \sum_{y,z \in B_1(0,7K\|x\|_1)} \P(\Gamma_6(z-y)^\complement).
\end{align*}
From Corollary~\ref{cor:jump}, this is bounded from above by a multiple of
$\|x\|_1^{4d}e^{-\Cr{jump2}t^\Ar{jump3}}$.
Therefore, we get the desired bound for $P(T_t(0^*,x^*) \not=T^*(0,x))$.

We next estimate $E[(T_t(0^*,x^*)-T^*(0,x))^2]^{1/2}$.
Schwarz's inequality implies
\begin{align*}
	&E\bigl[ (T_t(0^*,x^*)-T^*(0,x))^2 \bigr]\\
	&= E\bigl[ (T_t(0^*,x^*)-T^*(0,x))^2 \1{\{ T_t(0^*,x^*) \not=T^*(0,x) \}} \bigr]\\
	&\leq \bigl( E\bigl[ T_t(0^*,x^*)^4 \bigr]^{1/2}+E\bigl[ T^*(0,x))^4 \bigr]^{1/2} \bigr) P(T_t(0^*,x^*) \not=T^*(0,x))^{1/2}.
\end{align*}
By \eqref{eq:ttbd},
\begin{align*}
	E\bigl[ T_t(0^*,x^*)^4 \bigr]
	&\leq (4K)^4E\bigl[ (t \vee \|0^*-x^*\|_\infty)^4 \bigr]\\
	&\leq (12K)^4(2E[\|0^*\|_1^4]+1)\|x\|_1^4.
\end{align*}
On the other hand, letting $r(s):=s^{1/4}/(3\Cr{APtype1})$, one has
\begin{align*}
	E\bigl[ T^*(0,x))^4 \bigr]
	&\leq (3\Cr{APtype1}\|x\|_1)^4+\int_{(3\Cr{APtype1}\|x\|_1)^4}^\infty P\bigl( T^*(0,x)^4 \geq s \bigr)\,ds\\
	&\leq (3\Cr{APtype1}\|x\|_1)^4+2\int_{(3\Cr{APtype1}\|x\|_1)^4}^\infty P(\|0^*\|_1 \geq r(s))\,ds\\
	&\quad +	\int_{(3\Cr{APtype1}\|x\|_1)^4}^\infty
		\sum_{\substack{y \in B_1(0,r(s))\\ z \in B_1(x,r(s))}} P\bigl( T(y,z) \geq s^{1/4},\,y \in \mathcal{I} \bigr)\,ds.
\end{align*}
Proposition~\ref{prop:APtype} yields that $E[T^*(0,x))^4]$ is not greater than a multiple of $\|x\|_1^4$.
Combining these bounds with that for $P(T_t(0^*,x^*) \not=T^*(0,x))$,
we can derive the desired bound for $E[(T_t(0^*,x^*)-T^*(0,x))^2]^{1/2}$,
and the proof is complete.
\end{proof}

Before starting the proof of Proposition~\ref{prop:ttcon}, let us prepare some notation and lemmata.
For a given $x \in \Z^d \setminus \{0\}$ and $t>0$,
tile $\Z^d$ with copies of $(-t/2,t/2]^d$ such that each box is centered at a point in $\Z^d$
and each site in $\Z^d$ is contained in precisely one box.
We denote these boxes by $\Lambda_q$, $q \in \N$, and consider the random variables
\begin{align*}
	U_q:=((\omega(z))_{z \in \Lambda_q},(S_\cdot(z,\ell))_{z \in \Lambda_q,\ell \in \N}),\qquad q \in \N.
\end{align*}
Note that $(U_q)_{q=1}^\infty$ are independent and identically distributed.
Due to \eqref{eq:ttbd}, $T_t(0,x)$ depends only on states in some finite boxes $\Lambda_1,\dots,\Lambda_Q$,
and $T_t(0,x)$ can be regarded as a function of $(U_q)_{q=1}^Q$:
\begin{align*}
	Z:=T_t(0,x)=T_t(0,x,U_1,\dots,U_Q).
\end{align*}
In addition, let $(U'_q)_{q=1}^Q$ be independent copies of $(U_q)_{q=1}^Q$ and define
\begin{align*}
	Z'_q:=T_t(0,x,U_1,\dots,U_{q-1},U'_q,U_{q+1},\dots,U_Q),\qquad 1 \leq q \leq Q.
\end{align*}
Our main tools for the proof of Proposition~\ref{prop:ttcon} are Chebyshev's inequality and the following exponential versions of
the Efron--Stein inequality, refer the reader to \cite[Theorem~6.16]{BouLugMas13_book} and \cite[Lemma~3.2]{GarMar10}:
For any $\lambda, \theta>0$ with $\lambda\theta<1$,
\begin{align}\label{eq:BLMv}
	\log E[\exp\{ -\lambda(Z-E[Z]) \}]
	\leq \frac{\lambda\theta}{1-\lambda\theta} \log E\biggl[ \exp\biggl\{ \frac{\lambda V_-}{\theta} \biggr\} \biggr],
\end{align}
where
\begin{align*}
	V_-:=\sum_{q=1}^Q E\big[ (Z-Z'_q)_-^2 \big| U_1,\dots,U_Q \bigr].
\end{align*}
Furthermore, if there exist $\delta>0$, functions $(\phi_q)_{q=1}^Q$, $(\psi_q)_{q=1}^Q$ and $(g_q)_{q=1}^Q$
such that for all $1 \leq q \leq Q$,
\begin{align*}
	(Z-Z'_q)_- \leq \psi_q(U'_q),\quad
	(Z-Z'_q)_-^2 \leq \phi_q(U'_q) g_q(U_1,\dots,U_Q)
\end{align*}
and $E[e^{\delta \psi_q(U_q)} \phi_q(U_q)]<\infty$, then
for any $\lambda, \theta>0$ with $\lambda<\delta \wedge \theta^{-1}$,
\begin{align}\label{eq:BLMw}
	\log E[\exp\{ \lambda(Z-E[Z]) \}]
	\leq \frac{\lambda\theta}{1-\lambda\theta} \log E\biggl[ \exp\biggl\{ \frac{\lambda W}{\theta} \biggr\} \biggr],
\end{align}
where
\begin{align*}
	W:=\sum_{q=1}^QE\big[ e^{\delta \psi_q(U_q)} \phi_q(U_q) \bigr]g_q(U_1,\dots,U_Q).
\end{align*}

We use the following lemmata to estimate the right sides of \eqref{eq:BLMv} and \eqref{eq:BLMw}.

\begin{lem}\label{lem:perturb}
Write $\pi_t(0,x)=(0=x_0,x_1,\dots,x_m=x)$ for the finite sequence of $\Z^d$ that has
$T_t(0,x)=\sum_{i=0}^{m-1}\sigma_t(x_i,x_{i+1})$, chosen with a deterministic rule to break ties.
Moreover, let $R_q$ be the event that $\pi_t(0,x)$ intersects $\Lambda_q$.
Then, we have for $1 \leq q \leq Q$,
\begin{align}\label{eq:perturb}
	(Z-Z'_q)_- \leq 8Kt\1{R_q}.
\end{align}
\end{lem}
\begin{proof}
Since $(Z-Z'_q)_-=(Z'_q-Z)\1{\{ Z \leq Z'_q \} \cap R_q}$, we focus on the event $\{ Z \leq Z'_q \} \cap R_q$ from now on.
Let us first treat the case where $x \in \Lambda_q$.
Denote $i_0:=\min\{ 0 \leq i \leq m:x_i \in \Lambda_q \}$ and set $a:=x_{i_0}$.
Then, since $\|a-x\|_\infty \leq t$,
\begin{align*}
	Z'_q-Z
	\leq T_t(a,x,U_1,\dots,U_{q-1},U'_q,U_{q+1},\dots,U_Q)
	\leq 4Kt.
\end{align*}

For the case where $x \not\in \Lambda_q$, let us introduce the indices $i_1$ and $i_2$ as follows:
\begin{align*}
	&i_1:=\min\{ 0 \leq i \leq m-1:x_i \in \Lambda_q \},\\
	&i_2:=\max\{ 0 \leq i \leq m-1:x_i \in \Lambda_q \}.
\end{align*}
In addition, write $a:=x_{i_1}$, $b:=x_{i_2}$ and $c:=x_{i_2+1}$.
If $\|a-c\|_\infty \leq t$, then
\begin{align*}
	Z'_q-Z \leq T_t(a,c,U_1,\dots,U_{q-1},U'_q,U_{q+1},\dots,U_Q) \leq 4Kt.
\end{align*}
If $\|a-c\|_\infty>t$ and $\|b-c\|_\infty \leq t$, then
\begin{align*}
	Z'_q-Z
	&\leq T_t(a,c,U_1,\dots,U_{q-1},U'_q,U_{q+1},\dots,U_Q)\\
	&\leq 4K\|a-c\|_\infty\\
	&\leq 4K(\|a-b\|_\infty+\|b-c\|_\infty) \leq 8Kt.
\end{align*}
Otherwise (i.e., $\|a-c\|_\infty>t$ and $\|b-c\|_\infty>t$),
\begin{align*}
	Z'_q-Z
	&\leq T_t(a,c,U_1,\dots,U_{q-1},U'_q,U_{q+1},\dots,U_Q)-\sigma_t(b,c)\\
	&\leq 4K(\|a-c\|_\infty-\|b-c\|_\infty)\\
	&\leq 4K\|a-b\|_\infty \leq 4Kt.
\end{align*}
With these observations, $Z'_q-Z \leq 8Kt$ is valid in the case where $x \not\in \Lambda_q$,
and \eqref{eq:perturb} follows.
\end{proof}

\begin{lem}\label{lem:hit}
There exists a constant $\Cl{hit1} \geq 1$ such that
\begin{align*}
	\sum_{q=1}^Q\1{R_q} \leq \Cr{hit1}K \biggl( 1 \vee \frac{\|x\|_\infty}{t} \biggr).
\end{align*}
\end{lem}
\begin{proof}
Let $\pi_t(0,x)=(0=x_0,x_1,\dots,x_m=x)$.
For each $z \in \Z^d$, write $w(z)$ for center of the box $\Lambda_q$ containing $z$.
Then, define $\rho_0:=0$ and for $j \geq 1$,
\begin{align*}
	\rho_{j+1}:=\min \bigl\{ \rho_j<i \leq m:x_i \not\in w(x_{\rho_j})+(-3t/2,3t/2]^d \bigr\},
\end{align*}
with the convention $\min\emptyset:=\infty$.
Denote $J:=\max\{ j \geq 1:\rho_j<\infty \}$ and we assume
\begin{align*}
	J>4K\biggl( 1 \vee \frac{\|x\|_\infty}{t} \biggr).
\end{align*}
By definition, we have $T_t(0,x) \geq Jt$ and hence
\begin{align*}
	T_t(0,x) \geq Jt>4K(t \vee \|x\|_\infty),
\end{align*}
which contradicts \eqref{eq:ttbd}.
Therefore,
\begin{align*}
	J \leq 4K\biggl( 1 \vee \frac{\|x\|_\infty}{t} \biggr),
\end{align*}
and the proof is complete since $\pi_t(0,x)$ intersects at most $3^d(J+1)$ $\Lambda_q$'s.
\end{proof}

We are now in a position to prove Proposition~\ref{prop:ttcon}.

\begin{proof}[\bf Proof of Proposition~\ref{prop:ttcon}]
Fix arbitrary $\gamma>0$, $x \in \Z^d \setminus \{0\}$ and $0 \leq t \leq \gamma\sqrt{\|x\|_1}$.
We use Chebyshev's inequality to obtain that for all $u,\lambda \geq 0$,
\begin{align}\label{eq:cheby}
\begin{split}
	&P(|T_t(0,x)-E[T_t(0,x)]| \geq u)\\
	&\leq e^{-\lambda u} E[\exp\{ \lambda|Z-E[Z]| \}]\\
	&\leq e^{-\lambda u}(E[\exp\{ -\lambda(Z-E[Z]) \}]+E[\exp\{ \lambda(Z-E[Z]) \}]).
\end{split}
\end{align}
On the other hand, Lemmata~\ref{lem:perturb} and \ref{lem:hit} show that
\begin{align*}
	V_- \leq (8Kt)^2 \sum_{q=1}^Q\1{R_q} \leq \Cr{hit1}(8Kt)^2 \biggl( 1 \vee \frac{\|x\|_\infty}{t} \biggr).
\end{align*}
Moreover, taking $\delta:=1/t$, $\phi_q:=(8Kt)^2$, $\psi_q:=8Kt$ and $g_q:=\1{R_q}$ (see the notation above \eqref{eq:BLMw}),
we use Lemmata~\ref{lem:perturb} and \ref{lem:hit} again to obtain
\begin{align*}
	W \leq (8Kt)^2 e^{8K} \sum_{q=1}^Q\1{R_q}
	\leq \Cr{hit1} (8Kt)^2 e^{8K} \biggl( 1 \vee \frac{\|x\|_\infty}{t} \biggr).
\end{align*}
These bounds combined with \eqref{eq:BLMv}, \eqref{eq:BLMw} and \eqref{eq:cheby} prove that
for all $u \geq 0$ and for all $\lambda,\theta>0$ with $0<\lambda<t^{-1} \wedge (2\theta)^{-1}$,
\begin{align*}
	&P(|T_t(0,x)-E[T_t(0,x)]| \geq u)\\
	&\leq 2e^{-\lambda u}
		\exp\biggl\{ \frac{\lambda^2}{1-\lambda \theta} 	\Cr{hit1} (8K)^2 e^{8K} t(t \vee \|x\|_\infty) \biggr\}\\
	&\leq 2\exp\bigl\{ 2\Cr{hit1} (8K)^2 e^{8K} t(t \vee \|x\|_\infty)\lambda^2-u\lambda \bigr\}.
\end{align*}
Substitute $u=t\sqrt{\|x\|_1}$ for 
\begin{align*}
	&P\Bigl( |T_t(0,x)-E[T_t(0,x)]| \geq t\sqrt{\|x\|_1} \Bigr)\\
	&\leq 2\exp\Bigl\{ 2\Cr{hit1} (8K)^2 e^{8K} t(t \vee \|x\|_\infty)\lambda^2-t\sqrt{\|x\|_1}\lambda \Bigr\}.
\end{align*}
To minimize the right side, we choose
\begin{align*}
	\lambda=\frac{\sqrt{\|x\|_1}}{4\Cr{hit1} (8K)^2 e^{8K}(t \vee \|x\|_\infty)}.
\end{align*}
Since $t \leq \gamma\sqrt{\|x\|_1}$, $\Cr{hit1} \geq 1$ and $K \geq \gamma$,
\begin{align*}
	\lambda \leq \frac{\sqrt{\|x\|_1}}{2K^2\|x\|_\infty}
	\leq \frac{\sqrt{\|x\|_1}}{2K\|x\|_1}=\frac{1}{2K\sqrt{\|x\|_1}}<\frac{1}{t}.
\end{align*}
In addition, taking $\theta=(3\lambda)^{-1}$ leads to $0<\lambda<t^{-1} \wedge (2\theta)^{-1}$.
Therefore,
\begin{align*}
	P\Bigl( |T_t(0,x)-E[T_t(0,x)]| \geq t\sqrt{\|x\|_1} \Bigr)
	\leq 2\exp\biggl\{ -\frac{t}{8\Cr{hit1} (8K)^2 e^{8K}(1+\gamma)} \biggr\},
\end{align*}
which proves the proposition.
\end{proof}



\end{document}